\documentclass[a4paper, 12pt]{article}
\usepackage{amsmath}
\usepackage{amssymb,amsthm}
\usepackage[runin]{abstract}

\usepackage{enumitem}%让itemize行间隔小
\setenumerate[1]{itemsep=0pt,partopsep=0pt,parsep=\parskip,topsep=5pt}
\setitemize[1]{itemsep=0pt,partopsep=0pt,parsep=\parskip,topsep=5pt}
\setdescription{itemsep=0pt,partopsep=0pt,parsep=\parskip,topsep=5pt}

\abslabeldelim{\ \ }
\newtheorem{theorem}{Theorem}[section]

\newtheorem{lemma}[theorem]{Lemma}
\newtheorem{remark}[theorem]{Remark}

\newenvironment{sequation}{\begin{equation}}{\end{equation}}
\newenvironment{sequation*}{\begin{equation*}}{\end{equation*}}

\numberwithin{equation}{section}
\title{High energy solutions of quadratic coupling Schr{\"o}dinger equation with nonconstant potential\thanks{Supported by NSFC (12171326, 12171014, 12001382) and KZ202010028048.}}

\author{Mingyang Han and Kai Zhang\thanks{Corresponding author. Email: myanghan@163.com.}\\ School of Mathematical Sciences, Capital Normal University}
\date{}
\begin{document}
	\maketitle %没有这个不显示标题
	\numberwithin{equation}{section}
	\setcounter{section}{0}
	\begin{abstract}
		In this paper we use the variational method, especially the perturbation method, to find high energy solutions of the Schr{\"o}dinger system
		\begin{sequation}
			\begin{cases} 
				{ - \Delta u + V(x)u = {\mu _1}\left| u \right|^{p-1}u + \beta uv,}&{x \in {\mathbb{R}^N},} \\ 
				{ - \Delta v + V(x)v = {\mu _2}\left| v \right|^{p-1}v + \frac{\beta }{2}{u^2},}&{x \in {\mathbb{R}^N},} \\ 
				{u,v \in {H^1}({\mathbb{R}^N}),}
			\end{cases}
		\end{sequation}
		where $2<p<2^\ast -1:=\frac{N+2}{N-2}$, $\mu_1,\mu_{2},\beta >0$, and the potential $V(x)$ satisfying some appropriate conditions. Moreover, the asymptotic behavior of the high energy solutions as $\beta\to 0$ is analyzed.
	\end{abstract}
	
	\paragraph{2000 MR Subject Classification} $\ \ $35I20; 35I50; 35I57
	\section{Introduction} %大标题
	%物理背景
	%normalized 解背景
	%方程特点

In this paper we use the variational method, especially the perturbation method, to find high energy solutions of the Schr{\"o}dinger system
\begin{sequation}\label{key}
	\begin{cases} 
		{ - \Delta u + V(x)u = {\mu _1}\left| u \right|^{p-1}u + \beta uv,}&{x \in {\mathbb{R}^N},} \\ 
		{ - \Delta v + V(x)v = {\mu _2}\left| v \right|^{p-1}v + \frac{\beta }{2}{u^2},}&{x \in {\mathbb{R}^N},} \\ 
		{u,v \in {H^1}({\mathbb{R}^N}),}
	\end{cases}
\end{sequation}
where $2<p<2^\ast -1:=\frac{N+2}{N-2}$, $\mu_1,\mu_{2},\beta >0$, and the potential $V(x)$ satisfying some appropriate conditions. Moreover, the asymptotic behavior of the high energy solutions as $\beta\to 0$ is analyzed.

System (\ref{key}) is a reduced system studied in \cite{colin2006numerical,colin2009stability,pomponio2010ground}. This system is related to Raman amplification in nonlinear optics. Raman amplification has many applications in nonlinear, ultrafast optics and optical communications. See \cite{headley2005raman,pomponio2010ground,tian2019schrodinger,wang2017solitary} for more information about system (\ref{key}).

Through out this paper, we denote $\int_{\mathbb{R}^N}  \cdot  dx  $ by $\int_{\mathbb{R}^N}  \cdot $ and $H:=H^1(\mathbb{R}^N) \times H^1(\mathbb{R}^N) $. In order to introduce the conditions that $V(x)$ satisfies, we set $V^{-}(x):=\max \{-V(x), 0\}$ and
\begin{sequation*}
S: = \inf \left\{ {|\nabla u|_{{L^2}\left( {{\mathbb{R}^N}} \right)}^2:u \in {\mathcal{D}^{1,2}}\left( {{\mathbb{R}^N}} \right),|u{|_{{L^{{2^*}}}({\mathbb{R}^N})}} = 1} \right\}.
\end{sequation*}
We assuming the following hypotheses on $V(x)$:
\begin{itemize}
	\item[$(V_0$)] $ V^{-} \in L^{N/2}(\mathbb{R}^N) $ and  $ \int_{\mathbb{R}^N} {\arrowvert V^-(x) \arrowvert} ^{\frac{N}{2}}   <S $,
\end{itemize}
\begin{itemize}
	\item[$(V_1$)]$	V(x)\leq V_\infty:=\mathop {\lim }\limits_{\left| x \right| \to  + \infty } V(x) \quad a.e. \ x\in \mathbb{R}^N$ with $V_\infty>0$.
\end{itemize}
According to \cite{batista2018positive}, ${\left( {\int_{{\mathbb{R}^N}} {\left( {{{\left| {\nabla u} \right|}^2} + V(x){u^2}} \right)} } \right)^{\frac{1}{2}}}$ is a norm of in $H^1(\mathbb{R}^N)$. Therefore, we define 
\begin{sequation*}
{\left\| u \right\|^2}: = \int_{{\mathbb{R}^N}} {\left( {{{\left| {\nabla u} \right|}^2} + V(x){u^2}} \right)} .
\end{sequation*}

Asymptotic potentials ($V_\infty:=\mathop {\lim }\limits_{\left| x \right| \to  + \infty } V(x)$) have been extensively studied. For example, asymptotic positive potentials which mean $\mathop {\inf }\limits_{{\mathbb{R}^N}} V(x) > 0$ have been considered in \cite{benci1987positive,rabinowitz1992class} for the study of single nonlinear Schr{\"o}dinger equations. The articles \cite{chen2021normalized,kurata2021asymptotic,liu2016positive,osada2022existence,pomponio2010ground} studied nonlinear Schr{\"o}dinger systems with positive asymptotic potential. Among them, \cite{kurata2021asymptotic},\cite{osada2022existence} and \cite{pomponio2010ground} studied nonlinear Schr{\"o}dinger system with three waves interaction, that is
\begin{sequation}\label{three}
	\begin{cases}
			{ - \Delta {u_1} + {V_1}(x){u_1} = {{\left| {{u_1}} \right|}^{p - 1}}{u_1} + \beta {u_2}{u_3}},&{\ in\ {\mathbb{R}^N}}, \\ 
			{ - \Delta {u_2} + {V_2}(x){u_2} = {{\left| {{u_2}} \right|}^{p - 1}}{u_2} + \beta {u_1}{u_3}},&{\ in\ {\mathbb{R}^N}}, \\ 
			{ - \Delta {u_3} + {V_3}(x)u_3 = {{\left| {{u_3}} \right|}^{p - 1}}{u_3} + \beta {u_1}{u_2}},&{\ in\ {\mathbb{R}^N}}, \\ 
			{{u_1},{u_2},{u_3} \in {H^1}({\mathbb{R}^N})}. 
	\end{cases}
\end{sequation}
\cite{pomponio2010ground} used the central compactness lemma obtained that (\ref{three}) has a vector ground state solution when $\left| \beta  \right|$ sufficiently large. Pomponio also showed that the ground state converges to the scalar ground state as $\gamma\to 0$. \cite{kurata2021asymptotic} added some results of \cite{pomponio2010ground} when $ \beta >0$ was sufficiently small. While \cite{osada2022existence} studied the normalized solutions of (\ref{three}). When considering the sign-change of potentials, refer to \cite{batista2018positive,batista2018solutions}. 

In this paper, $V(x)$ can be nonconstant, nonradial and sign-changing which results in loss of compactness. Inspired by \cite{peng2017multiple,wu2019multiple}, we use the compactness of single equations (\ref{single_1}) and (\ref{single_2}) to restore the compactness of the system (\ref{key}). 
\iffalse
We do not assume that $V(x)$ is radially symmetric and that $V(x)$ can have a negative part. This potential causes a certain degree of tightness loss, and it is inevitable to restore the compactness, which is often easier to achieve in the ground state solution (minimum energy solution of all solutions). But in many cases, only $(V_0)$ and $(V_1) $ can not restore the tightness, such as \cite{benci1987positive}. In this paper, we find a high energy solution of a system, and obtain the existence of the solution by using perturbation method and Minimax method. Our ideas mainly come from \cite{wu2019multiple} which considered a linearly coupled system with critical exponents and periodic potential. 
\fi
According to remark (\ref{two_ground_state}) below we know 
\begin{sequation}\label{single_1}
	- \Delta u + V(x)u = {\mu _1}\left| u \right|^{p-1}u,\quad u \in {H^1}\left( {{\mathbb{R}^N}} \right),
\end{sequation}
and \begin{sequation}\label{single_2}
	- \Delta u + V(x)u = {\mu _2}\left| u \right|^{p-1}u,\quad u \in {H^1}\left( {{\mathbb{R}^N}} \right),
\end{sequation}
respectively have a ground state solution which means the solution with the least energy among all solutions. We use the ground state solution $U$ of the equation (\ref{single_1}) and the ground state solution $V$ of the equation (\ref{single_2}) to find a perturbed solution of system (\ref{key}). Easy to know if $u_0$ is a solution of (\ref{single_1}), and $v_0$ is a solution of (\ref{single_2}) then $(u_0,0)$ and $(0,v_0)$ are two solutions of (\ref{key}) and ${I_\beta }({u_0},0),{I_\beta }(0,{v_0}) > 0$. Hence there is a $(u,v)\in H$ such that ${I_\beta }\left( {u,v} \right) \leq \min \left\{ {{I_\beta }\left( {{u_0},0} \right),{I_\beta }\left( {0,{v_0}} \right)} \right\}$. While the solution $ \left( {{{\overline u }_\beta },{{\overline v }_\beta }} \right)$ we find in this paper satisfies ${I_\beta }\left( {{{\overline u }_\beta },{{\overline v }_\beta }} \right) \to {I_\beta }\left( {{u_0},0} \right) + {I_\beta }\left( {0,{v_0}} \right) > {I_\beta }\left( {u,v} \right),\beta \to 0$, therefore $\left( {{{\overline u }_\beta },{{\overline v }_\beta }} \right)$ is a high energy solution.
\iffalse
Define 
\begin{sequation}
	{J^{\infty}_i}(u) = \frac{1}{2}\int_{{\mathbb{R}^N}} {\left( {{{\left| {\nabla u} \right|}^2} + {V_\infty }{u^2}} \right)}  - \frac{1}{3}\int_{{\mathbb{R}^N}} {{\mu _i}{{\left| u \right|}^3}} 
\end{sequation}
\fi

Define $\left\| u \right\|_\infty^2: = \int_{\mathbb{R}^N}  {\left( {{{\left| {\nabla u} \right|}^2} + {V_\infty}{u^2}} \right)}$, then $\left\| u \right\|_\infty$ is a norm of $H^1(\mathbb{R}^N)$ by $V_\infty >0$.
Through a standard calculation one can know that solutions of (\ref{key}) are the critical points of
\begin{sequation}
		\begin{aligned}
			{I_\beta }(u,v) =& \frac{1}{2}\int_{{\mathbb{R}^N}} {\left( {{{\left| {\nabla u} \right|}^2} + {V(x)}{u^2}} \right)}  + \frac{1}{2}\int_{{\mathbb{R}^N}} {\left( {{{\left| {\nabla v} \right|}^2} + {V(x)}{v^2}} \right)}  - \frac{1}{p}{\mu _1}\left| u \right|_p^p - \frac{1}{p}\mu_2\left| v \right|_p^p - \frac{{3\beta }}{2}\int_{{\mathbb{R}^N}} {{u^2}v} .
		\end{aligned}
\end{sequation}

\iffalse%多行注释
The corresponding Nehari manifold is
\begin{sequation*}
	\mathcal{N}: = \left\{ {(u,v) \in H\backslash \{ (0,0)\} :\left\langle {I_\beta ^\prime (u,v),(u,v)} \right\rangle  = 0} \right\},
\end{sequation*}
and we define
\begin{sequation}
	c_\beta:=\mathop {\inf }\limits_\mathcal{N} {I_\beta }(u,v)=\mathop {\inf }\limits_\mathcal{N}.
\end{sequation}
\fi
The following theorem is the main result of our paper.
\begin{theorem}\label{th1}
	Assume that $(V_0)$ and $(V_1)$ hold. For $\mu_i>0,i=1,2$ and $\beta>0$, there exists $\beta _0 >0$ such that system (\ref{key}) has a nontrivial solution $\left(\overline{u}_\beta , \overline{v}_\beta \right)$ if $0<\beta <\beta _0$. Let $\beta _n \in\left(0, \beta _0\right)$ be a sequence with $\beta _n \rightarrow 0$ as $n \rightarrow+\infty$. Then passing to a subsequence if necessary, $\left(\overline{u}_{\beta _n}, \overline{v}_{\beta _n}\right) \rightarrow\left(U, V\right)$ in $H$, where $U$ and $V$ are ground states of (\ref{single_1}) and (\ref{single_2}) respectively. 
\end{theorem} 
	
In Section 2, we present some necessary notations and give some preparations on problem (\ref{key}). The proof of Theorem \ref{th1} will be given in Section 3. Throughout the paper we use the notation $\left|  \cdot  \right|_p$ to denote the norm of $L^p(\mathbb{R}^N)$ and $\rightharpoonup$ to donate weak convergence in $H^1(\mathbb{R}^N)$ or $H$. 
	
\section{Some preparations} %大标题
In this section, we collect several results which will be used in the rest discussion. First, let's recall the basic Brezis-Lieb inequality in \cite{brezis1983relation}.
\begin{lemma}\label{B-L_lemma}
	Suppose $\{u_n\} \in H^1(\mathbb{R}^N)$ is a bounded sequence, $u_n \rightharpoonup u $ in $H^1(\mathbb{R}^N)$, then for $2<s<2^\ast$, we have
	\begin{sequation*}
		\begin{gathered}
			\mathop {\lim }\limits_{n \to \infty } \int_{{\mathbb{R}^N}} {\left( {{{\left| {\nabla {u_n}} \right|}^2} - {{\left| {\nabla u} \right|}^2} - {{\left| {\nabla \left( {{u_n} - u} \right)} \right|}^2}} \right)}  = 0 \hfill ,
			\mathop {\lim }\limits_{n \to \infty } \int_{{\mathbb{R}^N}} {\left( {{{\left| {{u_n}} \right|}^s} - {{\left| u \right|}^s} - {{\left| {{u_n} - u} \right|}^s}} \right)}  = 0 ,\hfill \\
			\mathop {\lim }\limits_{n \to \infty } \int_{{\mathbb{R}^N}} {\left( {V(x)u_n^2 - V(x){u^2} - V(x){{\left( {{u_n} - u} \right)}^2}} \right)}  = 0. \hfill \\ 
		\end{gathered} 
	\end{sequation*}
\end{lemma}
\begin{proof}
		We only prove the last formula, others can be proved by Theorem 1 in \cite{brezis1983relation}. Let $ j(u)=V(x)u^2$ and $g_n={u_n} - u$ then 
	$$
	\begin{aligned}
		\left| {j(a + b) - j(a)} \right| &= \left| {2V(x)ab + V(x){b^2}} \right| \hfill \\&
		\leq  \varepsilon V(x){a^2} + \left( {1 + \frac{1}{\varepsilon }} \right)V(x){b^2} \hfill \\&
		= \varepsilon {\varphi _\varepsilon }(a) + {\psi _\varepsilon }(b), \hfill \\ 
	\end{aligned} 
	$$
	We have
	$$
	\begin{gathered}
		\int_{{\mathbb{R}^N}} {{\varphi _\varepsilon }({g_n}) = \int_{{\mathbb{R}^N}} {V(x){{\left( {{u_n} - u} \right)}^2}} }  \leq  C{\left\| {{u_n}} \right\|^2} \leq  \widetilde C, \hfill \\
		\int_{{\mathbb{R}^N}} {{\psi _\varepsilon }(u)}  = \int_{{\mathbb{R}^N}} {\left( {1 + \frac{1}{\varepsilon }} \right)V(x){u^2}}  < \infty \ for \ all\ \varepsilon  > 0. \hfill \\ 
	\end{gathered} 
	$$
	According to Theorem 2 in \cite{brezis1983relation}, the last formula is proved.
	
\end{proof}

Let 
\begin{sequation}
\overline S_p : = \mathop {\inf }\limits_{u \in {H^1}({\mathbb{R}^N}),{{\left| u \right|}_p} = 1} \int_{{\mathbb{R}^N}} {\left( {{{\left| {\nabla u} \right|}^2} + {V_\infty }{u^2}} \right)} ,
\end{sequation}
It is well known that $\overline S_p$ can be obtained by some $u \in H^1(\mathbb{R}^N)$ (see, for example \cite[Theorem1.34]{willem1997minimax}). 

Firstly, we define the corresponding energy functionals of single equations (\ref{single_1}) and (\ref{single_2})
\begin{sequation*}
	J_i(u):=\frac{1}{2}\int_{{\mathbb{R}^N}} {\left( {{{\left| {\nabla u} \right|}^2} + V(x){u^2}} \right) - \frac{1}{p}\int_{{\mathbb{R}^N}}\mu_i {{{\left| u \right|}^p}} }, \quad i=1,2.
\end{sequation*}
\iffalse
and the corresponding limit functional
\begin{sequation*}
	J_{\infty,i}(u):=\frac{1}{2}\int_{{\mathbb{R}^N}} {\left( {{{\left| {\nabla u} \right|}^2} + {V_\infty }{u^2}} \right) - \frac{1}{p}\int_{{\mathbb{R}^N}}\mu_i {{{\left| u \right|}^p}} },\quad i=1,2.
\end{sequation*}
\fi
\iffalse
and
\begin{sequation*}
	{\mathcal{N}_{\infty,i}}: = \left\{ {u \in {H^1}({\mathbb{R}^N})\backslash \{ 0\} :\left\langle {J_{\infty,i}^\prime (u),u} \right\rangle  = 0} \right\},\quad i=1,2.
\end{sequation*}
\fi
Let $c_i^*: = \mu_i\left( {\frac{1}{2} - \frac{1}{p}} \right){\left( {\frac{\overline S_p}{{{\mu _i}}}} \right)^{\frac{p}{{p - 2}}}},i=1,2$. Inspired by \cite{benci1987positive,ding1989existence}, we have the following lemma.
\begin{lemma}\label{compactness1}
	Under assumption $(V_0)$ and $(V_1)$, any sequence $\{(u_n)\}\in H^1(\mathbb{R}^N)$, such that
	\begin{sequation}
		d:=\mathop {\sup }\limits_n {\text{ }}{J_{i}}({u_n}) < c^\ast,\quad J_{i}^\prime ({u_n}) \to 0,\quad i=1,2
	\end{sequation}
	contains a convergent subsequence.
\end{lemma}
\begin{proof}
	\iffalse
	According to 
\begin{sequation*}
\left\{ {\begin{array}{*{20}{c}}
		{{J_i}({u_n}) = \frac{1}{2}\int_{{\mathbb{R}^N}} {\left( {{{\left| {\nabla {u_n}} \right|}^2} + V(x)u_n^2} \right) - \frac{1}{p}\int_{{\mathbb{R}^N}} {{\mu _i}{{\left| {{u_n}} \right|}^p} < c^\ast} } }, \\ 
		{o\left( {\left\| {{u_n}} \right\|} \right) = \int_{{\mathbb{R}^N}} {\left( {{{\left| {\nabla {u_n}} \right|}^2} + V(x)u_n^2} \right) - \int_{{\mathbb{R}^N}} {{\mu _i}{{\left| {{u_n}} \right|}^p}} } } ,
\end{array}} \right.
\end{sequation*}
\fi
It's easy to know $\left\{ {\left( {{u_n}} \right)} \right\}$ is bounded in $H^1(\mathbb{R}^N)$. Assume $\left\| u \right\| \leq C_1$ and ${\left| u \right|_p}\leq { C_2}, 2\leq p\leq 2^\ast$. Going if necessary to a subsequence, we can assume that
\begin{sequation*}
	\begin{aligned}
		{u_n} \rightharpoonup u\  in{\text{ }}{H^1}\left( {{\mathbb{R}^N}} \right);\ {u_n} \to u\  in \ L_{loc}^p\left( {{\mathbb{R}^N}} \right);\ {u_n} \to u\  a.e.\ on{\text{ }}{\mathbb{R}^N}.
	\end{aligned}
\end{sequation*}
By a standard argument, we have 
\iffalse
\begin{sequation*}
	{\left| {{u_n}} \right|^p} \to {\left| u \right|^p}\quad in \ L_{loc}^{\frac{p}{p-1}}\left( {{\mathbb{R}^N}} \right),
\end{sequation*}	
and hence 
\begin{sequation*}
	- \Delta u + V(x)u = \mu_i \left| u \right|u,
\end{sequation*}
that is 
\fi
\begin{sequation}\label{weak_solution}
	{\int_{{\mathbb{R}^N}} {\left( {{{\left| {\nabla {u}} \right|}^2} + V(x)u^2} \right) = \int_{{\mathbb{R}^N}} {{\mu _i}{{\left| {{u}} \right|}^p}} } }.
\end{sequation}
Assuming $w_n=u_n-u$, we claim that $\int_{{\mathbb{R}^N}} {\left( {{V_\infty } - V(x)} \right)} w_n^2\to 0$.
In fact, by $V_\infty:=\mathop {\lim }\limits_{\left| x \right| \to  + \infty } V(x)$, there is a $R>0$ such that $\left| {{V_\infty } - V(x)} \right| < \frac{\varepsilon }{2{ C_2}},x \in {\mathbb{R}^N}\backslash {B_R}(0)$. Then 
\begin{sequation*}
\int_{{\mathbb{R}^N}\backslash {B_R}(0)} {\left| {\left( {{V_\infty } - V(x)} \right)w_n^2} \right|} <\frac{\epsilon}{2}.
\end{sequation*}
Hence combing with $w_n \to 0$ in $L^2_{loc}(\mathbb{R}^N)$, the claim is proved.

Assuming $J_i(u_n)\to c \leq d$, Lemma \ref{B-L_lemma} and the claim above leads to 
\begin{sequation}\label{value}
	\begin{aligned}
		{J_i}({u_n}) 
		= {J_i}(u) + \frac{1}{2}{\left\| {{w_n}} \right\|_\infty^2} - \frac{1}{p}\int_{{\mathbb{R}^N}}\mu_i {{{\left| {{w_n}} \right|}^p}}  + o(1) \to c.
	\end{aligned} 
\end{sequation}
However, by
\begin{sequation*}
{J_i}(u) = \frac{1}{2}\int_{{\mathbb{R}^N}} {\left( {{{\left| {\nabla u} \right|}^2} + V(x){u^2}} \right)}  - \frac{1}{p}\int_{{\mathbb{R}^N}} {{\mu _i}{{\left| u \right|}^p}}  = \left( {\frac{1}{2} - \frac{1}{p}} \right)\int_{{\mathbb{R}^N}} {{\mu _i}{{\left| u \right|}^p}}  \geq 0
\end{sequation*}
and (\ref{value}) we obtain
\begin{sequation}\label{threshold}
	\frac{1}{2}{\left\| {{w_n}} \right\|_\infty^2} - \frac{1}{p}\int_{{\mathbb{R}^N}} \mu_i{{{\left| {{w_n}} \right|}^p}} \leq c.
\end{sequation}

Since $\left\langle {J_i^\prime ({u_n}),{u_n}} \right\rangle  \to 0$, by (\ref{weak_solution}) we obtain 
\begin{sequation*}
	\left\| {{w_n}} \right\|_\infty^2- \int_{{\mathbb{R}^N}} {{\mu _i}{{\left| {{w_n}} \right|}^p}}  \to \left\langle {J_i^\prime (u),u} \right\rangle  = 0.
\end{sequation*}
Passing to a subsequence if necessary we may therefore assume that
\begin{sequation*}
\left\| {{w_n}} \right\|_\infty^2 \to b,\quad \int_{{\mathbb{R}^N}} {{\mu _i}{{\left| {{w_n}} \right|}^p}} \to b.
\end{sequation*}
By the Sobolev inequality, we have ${\left\| {{w_n}} \right\|^2} \geq \overline S_p \left| {{w_n}} \right|_p^2$ and so $b\geq \overline S_p {\left( {\frac{b}{{{\mu _i}}}} \right)^{\frac{2}{p}}}$. Euther $b=0$ or $b \geq \mu_i{\left( {\frac{\overline S_p}{{{\mu _i}}}} \right)^{\frac{p}{{p - 2}}}}$. If $b=0$, the proof is complete. Assume $b \geq \mu_i{\left( {\frac{\overline S_p}{{{\mu _i}}}} \right)^{\frac{p}{{p - 2}}}}$, by (\ref{threshold}) we obtain from
\begin{sequation*}
c_i^* = \mu_i\left( {\frac{1}{2} - \frac{1}{p}} \right){\left( {\frac{{\overline S_p }}{{{\mu _i}}}} \right)^{\frac{p}{{p - 2}}}} \leq \left( {\frac{1}{2} - \frac{1}{p}} \right)b \leq c \leq d < {c_i^*},
\end{sequation*}
a contradiction.
\end{proof}
Set $c_{i}:=\mathop {\inf }\limits_{u \in {\mathcal{N}_i}} {J_{i}}(u)$ where $\setlength{\abovedisplayskip}{3pt}
	\setlength{\belowdisplayskip}{3pt}
	{\mathcal{N}_{i}}: = \left\{ {u \in {H^1}({\mathbb{R}^N})\backslash \{ 0\} :\left\langle {J_{i}^\prime (u),u} \right\rangle  = 0} \right\}, i=1,2.$ Then we have the following lemma.
\begin{lemma}\label{shexianzuida}
	Under assumption $(V_0)$ and $(V_1)$,
	\begin{sequation}\label{shexian}
		c_i=\mathop {\inf }\limits_{H\backslash \{ 0\} } \mathop {\max }\limits_{t > 0} {J_i }\left( {tu} \right)>0.
	\end{sequation}
\end{lemma}
\begin{proof}
	 Let
\begin{sequation*}
h(t) = {J_i}(tu) = \frac{{{t^2}}}{2}\left\| u \right\|^2 - \frac{{{t^p}}}{p}{\mu _i}\left| u \right|_p^p.
\end{sequation*}
Then for $t>0$, ${h^\prime }(t) = 0 \Leftrightarrow tu \in \mathcal{N}$.
For any $u\in H^1(\mathbb{R}^N)$ with $u\ne 0$, there exist a unique $t_i>0$ such that
\begin{sequation*}
	\mathop {\max }\limits_{t > 0} h(t) = \mathop {\max }\limits_{t > 0} {J_i}(tu) = {J_i}({t_i}u) = \left( {\frac{1}{2} - \frac{1}{p}} \right)t_i^p{\mu _i}\left| u \right|_p^p,
\end{sequation*}
where ${t_i} = {\left( {\frac{{{{\left\| u \right\|}^2}}}{{{\mu _i}\left| u \right|_p^p}}} \right)^{\frac{1}{{p - 2}}}} > 0$. Then, it is easy to deduce (\ref{shexian}).
\end{proof}

Define 
\begin{sequation}\label{X_i}
	{X_i} = \left\{ {u \in {H^1}\left( {{\mathbb{R}^N}} \right):J_i^\prime (u) = 0,{J_i}(u) = {c_i}} \right\},i=1,2,
\end{sequation}
When $u\in X_i$, a simple calculation shows that $\left\| u \right\| = \sqrt {\frac{{2p{c_i}}}{{p - 2}}} ,i = 1,2$. We have the following lemma.
\begin{lemma}\label{compact}
	There holds $c_i<c^\ast$, and therefore $X_i$ is compact in $H^1(\mathbb{R}^N)$.
\end{lemma}
\begin{proof}
	Let $v>0$ be a minimizing function for $\overline S_p$. Hence
\begin{sequation}
	\int_{{\mathbb{R}^N}} {\left( {{{\left| {\nabla v} \right|}^2} + V(x){v^2}} \right)}  < \int_{{\mathbb{R}^N}} {\left( {{{\left| {\nabla v} \right|}^2} + {V_\infty }{v^2}} \right)} .
\end{sequation}
It follows that
\begin{sequation}
	\begin{aligned}
		0 < \mathop {\max }\limits_{t > 0} {J_i}(tv) &= \mathop {\max }\limits_{t > 0} \left( {\frac{{{t^2}}}{2}\int_{{\mathbb{R}^N}} {\left( {{{\left| {\nabla v} \right|}^2} + V(x){v^2}} \right)}  - \frac{{{t^p}}}{p}{\mu _i}\left| v \right|_p^p} \right) \hfill \\
		& = \left( {\frac{1}{2} - \frac{1}{p}} \right)\frac{1}{{\mu _i^{\frac{2}{{p - 2}}}}}{\left( {\frac{{\int_{{\mathbb{R}^N}} {\left( {{{\left| {\nabla v} \right|}^2} + V(x){v^2}} \right)} }}{{\left| v \right|_p^2}}} \right)^{\frac{p}{{p - 2}}}}\\
		&< \left( {\frac{1}{2} - \frac{1}{p}} \right)\frac{1}{{\mu _i^{\frac{2}{{p - 2}}}}}{\left( {\frac{{\int_{{\mathbb{R}^N}} {\left( {{{\left| {\nabla v} \right|}^2} + {V_\infty }{v^2}} \right)} }}{{\left| v \right|_p^2}}} \right)^{\frac{p}{{p - 2}}}} = {\mu _i}\left( {\frac{1}{2} - \frac{1}{p}} \right){\left( {\frac{{\bar S_p}}{{{\mu _i}}}} \right)^{\frac{p}{{p - 2}}}} = c_i^*. \\ 
	\end{aligned} 
\end{sequation}
By previous Lemma we can deduce $c_i<c^\ast$. Similar to Lemma \ref{compactness1}, we can get the compactness of $X_i$.

\end{proof}
\begin{remark}\label{two_ground_state}
	Lemma \ref{compact} leads to that (\ref{single_1}) has a ground state solution $U$ and (\ref{single_2}) has a ground state solution $V$.
\end{remark}

\section{Proof of Theorem \ref{th1}} %大标题
In this section, we give the proof of Theorem \ref{th1}. Firstly we start by proving several technical lemmas. Define 
\begin{sequation*}
X=X_1 \times X_2.
\end{sequation*}
According to Lemma \ref{compact} and Tychonoff' s theorem, we know $X$ is compact in $H$.	
\begin{lemma}\label{bounded}
	There exist constants $\overline C> \underline C  >0 $ such that
	\begin{sequation}
		\underline C \leq {\left\| u \right\|},{\left\| v \right\|} \leq \overline C,\ \forall (u,v) \in X.
	\end{sequation}

\end{lemma}
\begin{proof}
	It is obvious from the definition of $H$ and $X_i,i=1,2$.
\end{proof}
From the proof of Lemma \ref{shexianzuida} and the definition of $U, V$, we see that
\begin{sequation*}
c_1=J_1(U)=\max _{t>0} J_1(t U), \quad c_2=J_2(V)=\max _{s>0} J_2(s V),
\end{sequation*}
and there exist $0<t_1<1<t_2$ and $0<s_1<1<s_2$ such that
\begin{sequation}\label{4.1}
	J_1(t U) \leq \frac{c_1}{4},\quad \forall t \in\left(0, t_1\right] \cup\left[t_2,+\infty\right), \\
\end{sequation}
\begin{sequation}\label{4.2}
	J_2(s V) \leq \frac{c_2}{4}, \quad \forall s \in\left(0, s_1\right] \cup\left[s_2,+\infty\right).
\end{sequation}

Define $\overline{\gamma}: Q \mapsto H$ by
\begin{sequation*}
\overline{\gamma}(t, s):=\left(\overline{\gamma}_1(t), \overline{\gamma}_2(t)\right)=(t U, s V),
\end{sequation*}
where $Q:=(t, s) \in\left[0, t_2\right] \times\left[0, s_2\right]$. Then there exists $C_0>0$ such that $\mathop {\max }\limits_{(t,s) \in Q} \left\| {\bar \gamma (t,s)} \right\| \leq {C_0}.$
Recalling the constant $\overline C$ given in Lemma \ref{bounded}, we set
\begin{sequation*}
{c}_\beta:=\inf _{\gamma \in {\Gamma}} \max _{(t, s) \in Q}  I _\beta(\gamma(t, s)),\quad m_\beta:=\max _{(t, s) \in Q}  I _\beta(\overline{\gamma}(t, s)),
\end{sequation*}
where
\begin{sequation}\label{4.3}
\begin{aligned}
	{\Gamma}:=\{\gamma \in C(Q, H) \mid& \max _{(t, s) \in Q}\|\gamma(t, s)\| \leq 2 \overline C+C_0 \\&
	\left.\gamma(t, s)=\overline{\gamma}(t, s) \text { for }(t, s) \in Q \backslash\left\{\left(t_1, t_2\right) \times\left(s_1, s_2\right)\right\}\right\}.
\end{aligned}
\end{sequation}
It is obvious that ${\Gamma} \neq \emptyset $ since $\overline{\gamma}(t, s) \in {\Gamma}$.
\begin{lemma}\label{Minimax}
	$$
	\setlength{\abovedisplayskip}{3pt}
	\setlength{\belowdisplayskip}{3pt}
	\lim _{\beta \rightarrow 0^{+}} c_\beta=\lim _{\beta \rightarrow 0^{+}} m_\beta=c_0=c_1+c_2 .
	$$
\end{lemma}
\begin{proof}
On the one hand, noting $\beta>0$, we have $ I _\beta(\overline{\gamma}(t, s)) \leq  I _0(\overline{\gamma}(t, s))$, and so
\begin{sequation}\label{4.4}
\begin{aligned}
	m_\beta \leq m_0 &=\max _{(t, s) \in Q}  I _0(\overline\gamma(t, s)) \\
	&=\max _{t \in\left[0, t_2\right]} J_1\left(\overline\gamma_1(t)\right)+\max _{s \in\left[0, s_2\right]} J_2\left(\overline\gamma_2(s)\right)=J_1(U)+J_2(V)=c_1+c_2 .
\end{aligned}
\end{sequation}
Since $\overline{\gamma} \in {\Gamma}$, we have $ {c} _\beta \leq m_\beta$. Then it follows that
\begin{sequation}\label{4.5}
\mathop {\lim \sup }\limits_{\beta  \to {0^ + }} \;{c_\beta } \leq  \mathop {\lim \inf \;}\limits_{\beta  \to {0^ + }} {m_\beta } \leq  \mathop {\lim \sup }\limits_{\beta  \to {0^ + }} \;{m_\beta } \leq  {m_0},
\end{sequation}
and
\begin{sequation*}
 {c} _0 \leq m_0 .
\end{sequation*}
On the other hand, for any $\gamma(t, s)=\left(\gamma_1(t), \gamma_2(s)\right) \in {\Gamma}$, we define $r(\gamma):\left[t_1, t_2\right] \times\left[s_1, s_2\right] \mapsto \mathbb{R}^2$ by
\begin{sequation*}
r(\gamma)(t, s):=\left(J_3\left(\gamma_1(t)\right)-J_4\left(\gamma_2(s)\right), J_3\left(\gamma_1(t)\right)+J_4\left(\gamma_2(s)\right)-2\right),
\end{sequation*}
where $J_3, J_4: H \mapsto \mathbb{R}$ are defined by
\begin{sequation*}
\begin{aligned}
	&J_3(u):= \begin{cases}
		\frac{{\int_{{\mathbb{R}^N}} {\left( {{\mu _1}|u{|^p}} \right)} }}{{\int_{{\mathbb{R}^N}} {\left( {|\nabla u{|^2} + V(x){u^2}} \right)} }}, & \text { if } u \neq 0, \\
		0, & \text { if } u=0,\end{cases} \\
	&J_4(u):= \begin{cases}
		\frac{{\int {\left( {{\mu _2}|u{|^p}} \right)} }}{{\int_{{\mathbb{R}^N}} {\left( {|\nabla u{|^2} + V(x){u^2}} \right)} }} & \text { if } u \neq 0, \\
		0, & \text { if } u=0,\end{cases}
\end{aligned}
\end{sequation*}
respectively. By embedding inequality, it is easy to get that $J_3, J_4$ are continuous. Moreover,
\begin{sequation*}
\begin{aligned}
	r(\overline{\gamma})(t, s)=&\left( {\frac{{t\int_{{\mathbb{R}^N}} {{\mu _1}{{\left| U \right|}^p}} }}{{\int_{{\mathbb{R}^N}} {\left( {|\nabla U{|^2} + V(x){U^2}} \right)} }} - \frac{{s\int_{{\mathbb{R}^N}} {{\mu _2}{{\left| V \right|}^p}} }}{{\int_{{\mathbb{R}^N}} {\left( {|\nabla V{|^2} + V(x){V^2}} \right)} }}} \right. ,\\
	&\left. {\frac{{t\int_{{\mathbb{R}^N}} {{\mu _1}{{\left| U \right|}^p}} }}{{\int_{{\mathbb{R}^N}} {\left( {|\nabla U{|^2} + V(x){U^2}} \right)} }} + \frac{{s\int_{{\mathbb{R}^N}} {{\mu _2}{{\left| V \right|}^p}} }}{{\int_{{\mathbb{R}^N}} {\left( {|\nabla V{|^2} + V(x){V^2}} \right)} }} - 2} \right).
\end{aligned}
\end{sequation*}
Recalling the definition of $U, V$, we have $r(\overline{\gamma})(1,1)=0$. By direct calculations, we get
\begin{sequation*}
\operatorname{deg}\left(r(\overline{\gamma}),\left[t_1, t_2\right] \times\left[s_1, s_2\right],(0,0)\right)=1 .
\end{sequation*}
From (\ref{4.1}), (\ref{4.2}), we know that for any $(t, s) \in \partial \left(\left[t_1, t_2\right] \times\left[s_1, s_2\right]\right), r(\gamma)(t, s)=r(\overline\gamma)(t, s) \neq(0,0)$, so $\operatorname{deg}\left(r(\gamma),\left[t_1, t_2\right] \times\left[s_1, s_2\right],(0,0)\right)$ is well defined and
\begin{sequation*}
\operatorname{deg}\left(r(\gamma),\left[t_1, t_2\right] \times\left[s_1, s_2\right],(0,0)\right)=\operatorname{deg}\left(r(\overline{\gamma}),\left[t_1, t_2\right] \times\left[s_1, s_2\right],(0,0)\right)=1.
\end{sequation*}
Thus, there is $\left(t_0, s_0\right) \subset\left[t_1, t_2\right] \times\left[s_1, s_2\right]$ such that $r(\gamma)\left(t_0, s_0\right)=(0,0)$, that is, $J_3\left(\gamma_1\left(t_0\right)\right)=J_4\left(\gamma_2\left(s_0\right)\right)=1$, which indicates that $\gamma_1\left(t_0\right) \neq 0, \gamma_2\left(s_0\right) \neq 0$ and $\gamma_1\left(t_0\right) \in \mathcal{N}_1, \gamma_2\left(s_0\right) \in \mathcal{N}_2$. Then we deduce that
\begin{sequation}\label{4.6}
\max _{(t, s) \in Q}  I _0(\gamma(t, s))\geq I _0\left(\gamma\left(t_0, s_0\right)\right)=J_1\left(\gamma_1\left(t_0\right)\right)+J_2\left(\gamma_2\left(s_0\right)\right)\geq c_1+c_2=m_0.
\end{sequation}
Hence, $ {c} _0 \geq m_0$, which together with (\ref{4.5}) gives $ {c} _0=m_0$. To complete the proof, it suffices to show that
\begin{sequation}\label{4.7}
\liminf _{\beta \rightarrow 0+}  {c} _\beta \geq m_0 .
\end{sequation}
If not, there exist $\epsilon>0, \beta_n \rightarrow 0^{+}$and $\gamma_n=\left(\gamma_{1, n}, \gamma_{2, n}\right) \in \Gamma$ such that
\begin{sequation*}
\max _{(t, s) \in Q}  I _{\beta_n}\left(\gamma_n(t, s)\right) \leq m_0-2 \epsilon .
\end{sequation*}
From (\ref{4.3}) and embedding inequality, there exists $N_0>0$ sufficiently large such that
\begin{sequation*}
\mathop {\max }\limits_{(t,s) \in Q} \frac{{3{\beta _n}}}{2}\int {\left| {\gamma _{1,n}^2(t,s){\gamma _{2,n}}(t,s)} \right|}  \leq C{\beta _n} < \varepsilon ,\quad \forall n \geq {N_0}.
\end{sequation*}
Then it holds
\begin{sequation*}
\max _{(t, s) \in Q}  I _0\left(\gamma_n(t, s)\right) \leq \max _{(t, s) \in Q}  I _{\beta_n}\left(\gamma_n(t, s)\right)+\epsilon \leq m_0-\epsilon, \quad \forall n \geq N_0,
\end{sequation*}
which contradicts (\ref{4.6}). Therefore, (\ref{4.7}) holds, and the proof is complete.
\end{proof}
Denote
\begin{sequation*}
X^d:=\{(u, v) \in H: \operatorname{dist}((u, v), X) \leq d\}, \quad  I _\beta^c:=\left\{(u, v) \in H:  I _\beta(u, v) \leq c\right\} .
\end{sequation*}
\begin{lemma}\label{buchongyinli}
	Let $d>0$ be a fixed number and let $\{(u_n,v_n)\} \subset X^d$ be a sequence. Then, up to a subsequence, $(u_n,v_n) \rightharpoonup (u,v)\in X^{2d}$.
\end{lemma}
\begin{proof}
	By Lemma \ref{bounded} and the definition of $X^d$, there exists a sequence $\left\{\left(\hat{u}_n, \hat{v}_n\right)\right\} \subset X$ such that
\begin{sequation*}
	\operatorname{dist}\left(\left(u_n, v_n\right),\left(\hat{u}_n, \hat{v}_n\right)\right) \leq d .
\end{sequation*}
	Then we infer from the compactness of $X$ that $\left(\hat{u}_n, \hat{v}_n\right) \rightarrow(\hat{u}, \hat{v})$ in $H$ for some $(\hat{u}, \hat{v}) \in X$, so $\operatorname{dist}\left((\hat{u}, \hat{v}),\left(\hat{u}_n, \hat{v}_n\right)\right)$ $\leq d$ for $n>0$ large enough. Then
\begin{sequation*}
	\operatorname{dist}\left(\left(u_n, v_n\right),(\hat{u}, \hat{v})\right) \leq \operatorname{dist}\left(\left(u_n, v_n\right),\left(\hat{u}_n, \hat{v}_n\right)\right)+\operatorname{dist}\left((\hat{u}, \hat{v}),\left(\hat{u}_n, \hat{v}_n\right)\right) \leq 2 d .
\end{sequation*}
	Since $\left(u_n, v_n\right) \rightharpoonup(u, v)$ in $H$, we have
	\begin{sequation*}
	\left\| {(u - \hat u,v - \hat v)} \right\| \leq \mathop {\lim \inf }\limits_{n \to \infty } \left\| {\left( {{u_n} - \hat u,{v_n} - \hat v} \right)} \right\| \leq 2d,
	\end{sequation*}
	that is, $(u, v) \in B_{2 d}(\hat{u}, \hat{v}) \subset X^{2 d}$. The lemma holds.
\end{proof}
\begin{lemma}\label{PS_condition}
	Set ${d_1}: = \frac{1}{2}\mathop {\min }\limits_{i = 1,2} \left\{ {\sqrt {\frac{{2p{c_i}}}{{p - 2}}} } \right\}$ and let $d \in\left(0, d_1\right)$. Assume that there are sequences $\beta_n>0$ with $\beta_n \rightarrow 0$ as $n \rightarrow \infty$ and $\left\{\left(u_n, v_n\right)\right\} \subset X^d$ such that
	\begin{sequation}\label{tiaojian}
	\lim _{n \rightarrow \infty}  I _{\beta_n}\left(u_n, v_n\right) \leq  {c} _0, \quad \lim _{n \rightarrow \infty}  I _{\beta_n}^{\prime}\left(u_n, v_n\right)=0 .
	\end{sequation}
	Then passing to a subsequence, $\left(u_n, v_n\right) \rightarrow(u, v)$ in $H$ for some $(u, v) \in X$.
\end{lemma}
\begin{proof}
	By Lemma \ref{buchongyinli} we assume that $\left(u_n, v_n\right) \rightharpoonup(u, v)\in X^{2 d}$. By (\ref{X_i}) and the choice of $d_1$, we can deduce that $u \neq 0$ and $v \neq 0$. Considering that $\mathop {\lim }\limits_{n \to \infty } I_{{\beta _n}}^\prime \left( {{u_n},{v_n}} \right) = 0$ and $\left\{\left(u_n, v_n\right)\right\}$ is bounded, we deduce for all $(\varphi, \psi) \in H$,
\begin{sequation*}
	\begin{aligned}
		\left\langle I_0^{\prime}(u, v),(\varphi, \psi)\right\rangle=& \int_{\mathbb{R}^N}\left(\nabla u \nabla \varphi+V(x)  u \varphi\right)  +\int_{\mathbb{R}^N}\left(\nabla v \nabla \psi+V(x)  v \psi\right)   - \int_{{\mathbb{R}^N}} \mu_1{\left| u \right|^{p-1}u\varphi }  - \int_{{\mathbb{R}^N}}\mu_2 {\left| v \right|^{p-1}v\psi }    \\
		=& \mathop {\lim }\limits_{n \to \infty } \left[ {\left\langle {I_{{\beta _n}}^\prime \left( {{u_n},{v_n}} \right),(\varphi ,\psi )} \right\rangle  + \frac{{{\beta _n}}}{2}\int_{{\mathbb{R}^N}} {u_n^2} \varphi  + {\beta _n}\int_{{\mathbb{R}^N}} {{u_n}{v_n}} \psi } \right]\\
		=& 0 .
	\end{aligned}
\end{sequation*}
	Hence, $I_0^{\prime}(u, v)=0$, that is, $u$ and $v$ are solutions of (\ref{single_1}) and (\ref{single_2}) respectively. Since $U$ and $V$ are ground state solutions of (\ref{single_1}) and (\ref{single_2}), then 
	\begin{sequation}\label{lew_energy}
		I_0(u,v) \geq c_1+c_2=c_0.
	\end{sequation} 
	Furthermore, since $\left(u_n, v_n\right) \in X^d$ for all $n$, we get
\begin{sequation*}
	\left\langle {I_0^\prime \left( {{u_n},{v_n}} \right),(\varphi ,\psi )} \right\rangle  = \left\langle {I_{{\beta _n}}^\prime \left( {{u_n},{v_n}} \right),(\varphi ,\psi )} \right\rangle  + \frac{{{\beta _n}}}{2}\int_{{\mathbb{R}^N}} {u_n^2} \varphi  + {\beta _n}\int_{{\mathbb{R}^N}} {{u_n}{v_n}} \psi  = o(1){\left\| {(\varphi ,\psi )} \right\|_H}.
\end{sequation*}
	On the other hand, from (\ref{tiaojian}) we get
	\begin{sequation}\label{4.9}
	\begin{aligned}
		 {c} _0 & \geq \lim _{n \rightarrow \infty} I_{\beta_n}\left(u_n, v_n\right)=\lim _{n \rightarrow \infty} I_0\left(u_n, v_n\right)-\lim _{n \rightarrow \infty} \frac{\beta_n}{2} \int_{\mathbb{R}^N} u^2_n v_n   \\
		&=\lim _{n \rightarrow \infty} I_0\left(u_n, v_n\right):=m .
	\end{aligned}
	\end{sequation}
	So $\left\{\left(u_n, v_n\right)\right\}$ is a $(P S)_m$ sequence of $I_0$ with $m: = \mathop {\lim }\limits_{n \to \infty } {I_0}\left( {{u_n},{v_n}} \right)$. Thus, we have
\begin{sequation*}
	\begin{aligned}
		I_0(u, v)&=\left( {\frac{1}{2} - \frac{1}{p}} \right) \int_{\mathbb{R}^N}\left(|\nabla u|^2+V(x)  u^2\right)  +\left( {\frac{1}{2} - \frac{1}{p}} \right) \int_{\mathbb{R}^N}\left(|\nabla v|^2+V(x)  v^2\right)   \\
		& \leq \left( {\frac{1}{2} - \frac{1}{p}} \right) \liminf _{n \rightarrow \infty}\left[\int_{\mathbb{R}^N}\left(\left|\nabla u_n\right|^2+V(x)  u_n^2\right)  +\int_{\mathbb{R}^N}\left(\left|\nabla v_n\right|^2+V(x)  v_n^2\right)  \right] \\
		&=\liminf _{n \rightarrow \infty}\left(I_0\left(u_n, v_n\right)-\frac{1}{p}\left\langle I_0^{\prime}\left(u_n, v_n\right),\left(u_n, v_n\right)\right\rangle\right)=m ,
	\end{aligned}
\end{sequation*}
	which combing with (\ref{lew_energy}) and (\ref{4.9}) brings to $m=I_0(u, v)= {c} _0$. This indicates $\left(u_n, v_n\right) \rightarrow(u, v)$ strongly in $H$, which implies $(u, v) \in X$.
 \end{proof}
\begin{lemma}\label{Lower bound of gradient}
	Let $d_1$ be the constant given in Lemma \ref{PS_condition}. For a small constant $d \in\left(0, \frac{d_1}{2}\right)$, there exists $\delta \in(0,1)$ and $\beta_0 >0$ such that for any $\beta \in\left(0, \beta_0\right)$,
\begin{sequation*}
	\left\| I _\beta^{\prime}(u, v)\right\| \geq \delta, \quad \forall(u, v) \in  I _\beta^{m_\beta} \cap\left(X^d \backslash X^{\frac{d}{2}}\right) .
\end{sequation*}
\end{lemma}
\begin{proof}

By a contradiction, we assume that there exist sequences $\left\{\beta_n\right\}$ with $\mathop {\lim }\limits_{n \to \infty } {\beta _n} = 0$ and $\left\{\left(u_n, v_n\right)\right\} \subset$ $ I _{\beta_n}^{m _{\beta_n}} \cap\left(X^d \backslash X^{\frac{d}{2}}\right)$ such that $\left\| I _{\beta_n}^{\prime}\left(u_n, v_n\right)\right\| \rightarrow 0$ as $n \rightarrow \infty$. By Lemma \ref{Minimax}, we obtain $\mathop {\lim }\limits_{n \to \infty } {I_{{\beta _n}}}\left( {{u_n},{v_n}} \right) \leq {m_{{\beta _n}}} \leq {m_0} \leq {c_0}$. Then applying Lemma \ref{PS_condition}, there exists $(u, v) \in X$ such that $\left(u_n, v_n\right) \rightarrow(u, v)$ in $H$. Thus, $\left(u_n, v_n\right) \in X^{\frac{d}{2}}$ for $n$ large enough, which contradicts that $\left\{\left(u_n, v_n\right)\right\} \subset  I _{\beta_n}^{m_{ \beta_n}} \cap\left(X^d \backslash X^{\frac{d}{2}}\right)$. The proof is complete.
\end{proof}
In the rest, we fix a small $d \in\left(0, \frac{d_1}{2}\right)$ and the corresponding $\delta \in(0,1)$ and $\beta_0$ such that the conclusion of Lemma \ref{Lower bound of gradient} holds.
\begin{lemma}\label{energy_to_point_neberhood}
There is a $\beta_1 \in\left(0, \beta_0\right)$ and a $\sigma>0$ such that for any $\beta \in\left(0, \beta_1\right)$
$$
 I _\beta(\overline{\gamma}(t, s)) \geq  {c} _\beta-\sigma \text { implies }\  \overline{\gamma}(t, s) \in X^{\frac{d}{2}} .
$$
\end{lemma}
\begin{proof}
Assume by a contradiction that there exist $\beta_n \rightarrow 0, \sigma_n \rightarrow 0$ and $\left(t_n, s_n\right) \in Q$ such that
$$
 I _{\beta_n}\left(\overline{\gamma}\left(t_n, s_n\right)\right) \geq  {c} _{\beta_n}-\sigma_n,\quad  \overline{\gamma}\left(t_n, s_n\right) \notin X^{\frac{d}{2}} .
$$
Up to a subsequence, we assume that $\left(t_n, s_n\right) \rightarrow(\tilde{t}, \tilde{s}) \in Q$. Then by Lemma \ref{Minimax} and the last inequality, we have
\begin{sequation*}
 I _0(\gamma(\tilde{t}, \tilde{s})) \geq \lim _{n \rightarrow \infty} c_{\beta_n}=c_1+c_2,
\end{sequation*}
which indicates $(\tilde{t}, \tilde{s})=(1,1)$. Therefore,
\begin{sequation*}
\lim _{n \rightarrow \infty}\left\|\overline{\gamma}\left(t_n, s_n\right)-\overline{\gamma}(1,1)\right\|=0.
\end{sequation*}
However, $\overline{\gamma}(1,1)=(U, V) \in X$, which contradicts $\overline{\gamma}\left(t_n, s_n\right) \notin X^{\frac{d}{2}}$. Hence, the conclusion holds.
\end{proof}

For the constants $d, \delta$ given in Lemma \ref{Lower bound of gradient} and $\sigma$ defined in lemma \ref{energy_to_point_neberhood}, we set 
\begin{sequation}\label{4.11}
\sigma_0:=\min \left\{\frac{\sigma}{2}, \frac{c_1}{4}, \frac{1}{8} d \delta^2\right\}
\end{sequation}
, then by Lemma \ref{Minimax}, there exists $\beta_2 \in\left( {0,{\beta _1}} \right]$ such that
\begin{sequation*}
\left|c_\beta-m_\beta\right|<\sigma_0,\left|c_\beta-\left(c_1+c_2\right)\right|<\sigma_0, \quad \forall \beta \in\left(0, \beta_2\right) .
\end{sequation*}
\begin{lemma}\label{PS_sequence}
	For given $\beta \subset\left(0, \beta_2\right)$, there exists $\left\{\left(u_m, v_n\right)\right\} \subset X^d \cap  I _\beta^{m_\beta}$ such that
	$$
	 I _\beta\left(u_n, v_n\right) \rightarrow 0 \text {, as } n \rightarrow \infty \text {. }
	$$
\end{lemma}
\begin{proof}
For a $\beta \in\left(0, \beta^*\right)$, assume by a contradiction that there is $0<l(\beta)<1$ such that $\left\| {I_\beta ^\prime (u,v)} \right\| \geqslant l(\beta )$ for all $(v, v) \subset X^d \cap  I _\beta^{m_\beta}$. 
Then there exists a pecudo-gradient vector field $g_\beta$ in $H$ which is defined on a neighborhood $M_\beta$ of $X^\beta \cap  I _\beta^{m _{\beta}}$ such that
\begin{sequation*}
\begin{aligned}
	\left\| {{g_\beta }(u,v)} \right\|& \leq 2\min \left\{ {1,\left\| {I_\beta ^\prime (u,v)} \right\|} \right\},\\ 
	\left\langle {I_\beta ^\prime (u,v),{g_\beta }(u,v)} \right\rangle & \geq \min \left\{ {1,\left\| {I_\beta ^\prime (u,v)} \right\|} \right\}\left\| {I_\beta ^\prime (u,v)} \right\|.
\end{aligned}
\end{sequation*}
Let $\eta_\beta$ be a Lipschitz continuous function on $H$ such that $0 \leq \eta_\beta \leq 1, \eta_\beta \equiv 1$ on $X^d \cap  I _\beta^{m_\beta}$ and $\eta_\beta \equiv 0$ on $H \backslash M_\beta$, and $\zeta_\beta$ be a Lipschitz continuous function on $\mathbb{R}$ such that $0 \leq \zeta_\beta(t) \leq 1, \zeta_\beta(t)\equiv 1$ if $\left|t-c_\beta\right| \leq \frac{\sigma}{2}$, and $\zeta_\beta(t)=0$ if $\left|t- {c} _\beta\right|>\sigma$. Let
\begin{sequation*}
f_\beta(u, v)=\left\{\begin{array}{l}
	-\eta_\beta(u, v) \zeta_\beta\left( I _\beta(u, v)\right) g_\beta(u, v),\quad (u, v) \in M_\beta, \\
	0, \quad (u, v) \in H \backslash M_\beta .
\end{array}\right.
\end{sequation*}
Then for the initial value problem
\begin{sequation*}
\left\{ {\begin{array}{*{20}{l}}
		{\frac{d}{{d\theta }}{\phi _\beta }(u,v,\theta ) = {f_\beta }\left( {{\phi _\beta }(u,v,\theta )} \right),} \\ 
		{{\phi _\beta }(u,v,0) = (u,v),} 
\end{array}} \right.
\end{sequation*}
there is a global solution $\phi_\beta: H \times[0,+\infty) \to H$. Moreover, $\phi_\beta$ has the following properties:
\begin{itemize}
	\item[$(i)$] $\phi_\beta(u, v, \theta)=(u, v)$ if $\theta=0$ or $(u, v) \subset H \backslash M_\beta$ or $\left| I _\beta(u, v)- {c} _\beta\right| \geq \sigma$.
	\item[$(ii)$] $\left\|\frac{d}{d \theta} \phi_\beta(u, v, \theta)\right\| \leq 2$.
	\item[$(iii)$] $\frac{d}{d \theta}  I _\beta\left(\phi_\beta(u, v, \theta)\right)=\left\langle I^\prime _\beta\left(\phi_\beta(u, v, \theta)\right), f_\beta\left(\phi_\beta(u, v, \theta)\right)\right\rangle \leq 0$.
\end{itemize}

Now we divide the proof into two steps.

\textbf{Step 1}: we claim that for any $(t, s) \in Q$, there is $\theta_0 \in[0,+\infty)$ such that
\begin{sequation*}
{\phi _\beta }\left( {\overline \gamma  (t,s),{\theta _0}} \right) \in I_\beta ^{{{ c }_\beta } - {\sigma _0}},
\end{sequation*}
where $\sigma_0$ is defined in (\ref{4.11}). In fact assume by a contradiction that there is $(t, s) \in Q$ such that
\begin{sequation*}
 I _\beta\left(\phi_\beta(\overline \gamma(t, s), \theta_0)\right)>c_\beta-\sigma_0, \quad \forall \theta \geq 0 .
\end{sequation*}
Then from (\ref{4.11}) and Lemma \ref{energy_to_point_neberhood}, we have $\overline \gamma(t, s) \in X^{\frac{d}{2}}$. Noting that
\begin{sequation*}
 I _\beta\left(\overline{\gamma}\left(t, s\right)\right) \leq m_\beta< {c} _\beta+\sigma_0,
\end{sequation*}
we know from property (iii) that
\begin{sequation*}
c_\beta-\sigma_0< I _\beta\left(\phi_\beta(\overline \gamma(t, s), \theta)\right) \leq m_\beta<\overline c_\beta+\sigma_0,\quad  \forall \theta \geq 0 .
\end{sequation*}
Then it follows that $\zeta_\beta\left( I _\beta\left(\phi_\beta(\overline \gamma(t, s), \theta)\right)\right)=1$. If $\left.\phi_\beta(\overline \gamma(t, s), \theta)\right) \in X^d$ for all $\theta \geq 0$, then
\begin{sequation*}
\eta_\beta\left(\phi_\beta(\overline \gamma(t, s), \theta)\right) \equiv 1,\left\| I _\beta^{\prime}\left(\phi_\beta(\overline \gamma(t, s), \theta)\right)\right\| \geq l(\beta), \quad \forall \theta>0 \text {. }
\end{sequation*}
Therefore,
\begin{sequation*}
{I_\beta }\left( {{\phi _\beta }\left( {\overline \gamma (t,s),\frac{\sigma }{{l{{(\beta )}^2}}}} \right)} \right) \leq  {{ c}_\beta } + \frac{\sigma }{2} - \int_0^{\frac{\sigma }{{l{{(\beta )}^2}}}} l {(\beta )^2}d\theta  \leq  {{c}_\beta } - \frac{\sigma }{2},
\end{sequation*}
which is a contradiction. Thus, there is a $\theta_0>0$ such that $\left.\phi_\beta(\overline \gamma(t, s), \theta)\right) \notin X^d$. Recalling that $\overline \gamma(t, s) \in X^{\frac{d}{2}}$, there is $0<\theta_1<\theta_2 \leq \theta_0$ such that $\phi_\beta\left(\overline {\gamma}(t, s), \theta_1\right) \in \partial X^{\frac{d}{2}}, \phi_\beta\left(\overline {\gamma}(t, s), \theta_2\right) \in \partial X^d$ and $\phi_\beta(\overline {\gamma}(t, s), \theta) \in X^d \backslash X^{\frac{d}{2}}$ for all $\theta \in\left(\theta_1, \theta_2\right)$. Then From Lemma \ref{Lower bound of gradient}
\begin{sequation*}
\left\| I _\beta^{\prime}\left(\phi_\beta(\overline{\gamma}(t, s), \theta)\right)\right\| \geq \delta, \forall \theta \in\left(\theta_1, \theta_2\right) .
\end{sequation*}
Applying property (ii), we have
\begin{sequation*}
\frac{d}{2} \leq\left\|\phi_\beta\left(\overline \gamma(t, s), \theta_1\right)-\phi_\beta\left(\overline \gamma(t, s), \theta_2\right)\right\| \leq 2\left|\theta_1-\theta_2\right| .
\end{sequation*}
Therefore, $\left|\theta_1-\theta_2\right| \geq \frac{d}{4}$, and it follows
\begin{sequation*}
\begin{aligned}
	 I _\beta\left(\phi_\beta\left(\overline{\gamma}(t, s), \theta_2\right)\right) & \leq  I _\beta\left(\phi_\beta\left(\overline{\gamma}(t, s), \theta_1\right)\right)+\int_{\theta_1}^{\theta_2} \frac{d}{d \theta}  I _\beta\left(\phi_\beta(\overline \gamma (t,s),\theta)\right) \\
	& \leq  {c} _\beta+\sigma_0-\delta^2\left(\theta_2-\theta_1\right) \leq  {c} _\beta+\sigma_0-\frac{1}{4} d \delta^2 \\
	& \leq  {c} _\beta-\sigma_0,
\end{aligned}
\end{sequation*}
which is also contradiction, so the claim holds.

From the first step, we define
\begin{sequation*}
H(t, s):=\inf \left\{\theta \geq 0:  I _\beta\left(\phi_\beta(\overline{\gamma}(t, s), \theta)\right) \leq  {c}_\beta-\sigma_0\right\},
\end{sequation*}
and $\gamma(t, s):=\phi_\beta(\gamma(t,s), H(t,s))$. Then $\phi_\beta(\gamma(t,s)) \leq c_\beta-\sigma_0$ for all $(t,s) \in Q$.

\textbf{Step 2}: we show that
\begin{sequation}\label{4.12}
\gamma(t,s) \in {\Gamma} .
\end{sequation}
For any $(t, s) \in Q \backslash\left(t_1, t_2\right) \times\left(s_1, s_2\right)$, we have
\begin{sequation*}
 I _\beta(\overline{\gamma}(t, s)) \leq  I _0(\overline{\gamma}(t, s))=J _1\left(\overline{\gamma}_1(t)\right)+ J _2\left(\overline{\gamma}_2(s)\right) \leq \frac{c_1}{4}+c_2 \leq c_1+c_2-3 \sigma_0< {c} _\beta-\sigma_0,
\end{sequation*}
which implies $H(t, s)=0$ and $\gamma(t, s)=\overline \gamma(t, s)$.

From the definition of ${\Gamma}$, it suffices to prove $\|\gamma(t, s)\| \leq 2 \overline C+C_0$ for all $(t, s) \in Q$ and $H(t, s)$ is continuous with respect to $(t, s) \in Q$.

For any $(t, s) \in Q$, if $ I _\beta(\overline{\gamma}(t, s)) \leq  {c} _\beta-\sigma_0$, then $H(t, s)=0$ and $\gamma(t, s)=\overline{\gamma}(t, s)$, and so $\|\gamma(t, s)\|=$ $\|\overline{\gamma}(t, s)\| \leq 2 \overline C+C_0$. If $ I _\beta(\overline{\gamma}(t, s))> {c} _\beta-\sigma_0$, then $\overline{\gamma}(t, s) \in X^{\frac{d}{2}}$ and
\begin{sequation*}
c_\beta-\sigma_0< I _\beta\left(\phi_\beta(\overline{\gamma}(t, s), \theta)\right) \leq m_\beta<c_\beta+\sigma_0, \forall \theta \in[0, H(t, s)] .
\end{sequation*}
From the definition of $\zeta_\beta$, we have $\zeta_\beta\left( I _\beta\left(\phi_\beta(\overline{\gamma}(t, s), \theta)\right)\right) \equiv 1$ for $\theta \in[0, H(t, s)]$. Now we claim
\begin{sequation*}
\left.\gamma(t, s)=\phi_\beta(\overline{\gamma}(t, s), I(s, t))\right) \in X^d .
\end{sequation*}
If not, by a similar way as in the first step, we see that there exist $0<\theta_1<\theta_2<H(t, s)$ such that $ I _\beta\left(\phi_\beta\left(\overline{\gamma}(t, s), \theta_2\right)\right) \leq  {c} _\beta-\sigma_0$, which contradicts the definition of $H(t, s)$. Thus, the claim holds, and there exists $(u, v) \in X$ such that
\begin{sequation*}
\|\gamma(s, t)-(u, v)\| \leq d \leq \frac{C_0}{2} .
\end{sequation*}
Using Lemma \ref{bounded}, we have
\begin{sequation*}
\|\gamma(s, t)\| \leq\|(u, v)\|+\frac{C_0}{2} \leq 2 \overline C+C_0 .
\end{sequation*}
Now we show the continuity of $H(t, s)$. Fix some $(\overline{t}, \overline{s}) \in Q$. If $ I _\beta(\gamma(\overline{t}, \overline{s}))< {c} _\beta-\sigma_0$, then $H(\overline{t}, \overline{s})=0$ and $ I _\beta(\overline\gamma(\overline{t}, \overline s))< c_\beta-\sigma_0$. By the continuity of $\overline\gamma$, there is $\tau>0$ such that
\begin{sequation*}
 I _\beta(\overline{\gamma}(t, s))< {c} _\beta-\sigma_0, \forall(t, s) \in(\overline{t}-\tau, \overline{t}+\tau) \times(\overline{s}-\tau, \overline{s}+\tau) \cap Q,
\end{sequation*}
that is, $H(t, s)=0, \forall(t, s) \in(\overline{t}-\tau, \overline{t}+\tau) \times(\overline{s}-\tau, s+\tau) \cap Q$. Therefore, $H$ is continuous at $(\overline{t}, \overline s)$. If $ I _\beta(\gamma(\overline{t}, \overline{s}))= {c} _\beta-\sigma_0$. Then from the previous proof we have $\gamma(t, \overline{s})=\phi_\beta(\overline{\gamma}(t, \overline{s}), H(\overline{t}, \overline{s})) \in X^d$, and so
\begin{sequation*}
\left\| I _\beta^{\prime}\left(\phi_\beta(\overline{\gamma}(t, \overline{s}), H(\overline{t}, \overline{s}))\right)\right\| \geq l(\beta)>0.
\end{sequation*}
Then for any $\omega>0,  I _\beta\left(\phi_\beta(\overline{\gamma}(t, \overline{s}), H(t, \overline{s})+\omega)\right)< {c} _\beta-\sigma_0$. Since $\phi_\beta$ is continuous, there is $\tau>0$ such that $ I _\beta\left(\phi_\beta(\overline{\gamma}(t, s), H(\overline{t}, \overline{s})+\omega)\right)< {c} _\beta-\sigma_0, \forall(t, s) \in(\overline{t}-\tau, \overline{t}+\tau) \times(\overline{s}-\tau, \overline{s}+\tau) \cap Q$, so $H(t, s) \leq H(\overline{t}, \overline{s})+\omega$ and
\begin{sequation}\label{4.13}
0 \leq \limsup _{(t, s) \rightarrow(\overline{t}, \overline{s})}  H(t, s) \leq  H(\overline{t}, \overline{s}) .
\end{sequation}
If $H(\overline{t}, \overline{s})=0$, we have
\begin{sequation*}
\lim _{(t, s) \rightarrow(\overline{t}, s)} H(t, s)=H(\overline{t}, \overline{s}) .
\end{sequation*}
If $H(\overline{t}, \overline{s})>0$, then for any $0<\omega< H(\overline{t}, \overline{s})$, we have $ I _\beta\left(\phi_\beta(\overline{\gamma}(\overline{t}, \overline{s}), H(\overline{t}, \overline{s})-\omega)\right)> {c} _\beta-\sigma_0$. By the continuity of $\phi_\beta$ again, we get
\begin{sequation*}
\liminf _{(t, s) \rightarrow(\overline{t}, \overline{s})} H(t, s) \geq H(\overline{t}, s),
\end{sequation*}
which together with (\ref{4.13}) indicates that $H(s, t)$ is continuous at $(\overline{t}, \overline{s})$. Therefore, (\ref{4.12}) holds.

Now we have proved that $\gamma(t, s) \in {\Gamma}$ and $\mathop {\max }\limits_{(t,s) \in Q} {I_\beta }(\gamma (t,s)) \leq {{ c}_\beta } - {\sigma _0}$, which contradicts the definition of $ c_\beta$. The proof is complete.
\end{proof}
\paragraph{Proof of Theorem \ref{th1}.}Suppose ${d_1}= \frac{1}{2}\mathop {\min }\limits_{i = 1,2} \left\{ {\sqrt {\frac{{2p{c_i}}}{{p - 2}}} } \right\}$. According to Lemma \ref{PS_sequence} there is a $\beta_2 >0$ such that for any fixed $\beta \in\left(0, \beta_2\right)$, there is a sequence $\left\{\left(u_n^\beta, v_n^\beta\right)\right\} \in X^d$ such that
\begin{sequation*}
 I _\beta\left(u_n^\beta, v_n^\beta\right) \leq m_\beta,  \quad I _\beta^{\prime}\left(u_n^\beta, v_n^\beta\right) \rightarrow 0 ,n\to \infty.
\end{sequation*}
Then from Lemma \ref{buchongyinli}, there is a $\left(\overline u_\beta, \overline v_\beta\right) \in X^{2 d}$ such that $\left(u_n^\beta, v_n^\beta\right) \rightharpoonup \left(\overline u_\beta, \overline v_\beta\right)$ in $H$. Through a standard argument, we have $ I _\beta^{\prime}\left(\overline u_\beta, \overline v_\beta\right)=0$. Moreover, by the choice of $d$, we get $\overline u_\beta \neq 0,\overline v_\beta \neq 0$. Hence, $\left(\overline u_\beta, \overline v_\beta\right)$ is a nontrivial solution to (1.1).

Let $\beta_n \in\left(0, \beta_2\right)$ be a sequence with $\beta_n \rightarrow 0$ as $n \rightarrow \infty$. Then repeating the proof of Lemma \ref{PS_condition}, we have $\left(u_{\beta_n}, v_{\beta_n}\right) \rightarrow(\tilde{u}, \tilde{v})$ in $H$, where $(\tilde{u}, \tilde{v}) \in X$, that is, $\tilde{u}$ is a ground state of (\ref{single_1}) and $\tilde{v}$ is a ground state of (\ref{single_2}) respectively. The proof is complete.$\hfill\Box$

\iffalse
\section*{Acknowledgments}
	
	This work is supported by the National Natural Science Foundation of China (Nos. 12171326, 12001382, 12171014) and the Beijing Municipal Education Commission-Natural Science Foundation, People’s Republic of China (KZ202010028048).
\fi

	\small
	\bibliographystyle{plain}
	\bibliography{Other_mulitiple_quadratic}
\end{document}